\documentclass[12pt]{article}
\usepackage{amssymb, amsmath,amsfonts, latexsym, a4}
\usepackage[latin1]{inputenc}
\usepackage{graphicx}
\usepackage[all]{xy}

\usepackage[usenames]{xcolor}
\usepackage{ulem}

\newenvironment{proof}[1][Proof]{\noindent\textbf{#1.} }{\ \rule{0.5em}{0.5em}}

\newtheorem{definition}{Definition}[section]
\newtheorem{theorem}[definition]{Theorem}
\newtheorem{proposition}[definition]{Proposition}
\newtheorem{corollary}[definition]{Corollary}
\newtheorem{lemma}[definition]{Lemma}
\newtheorem{example}[definition]{Example}

\newcommand{\AWB}{\sf AWB}
\newcommand{\as}{\sf A}
\newcommand{\he}{\sf H}

\newcommand{\me}{\sf M}

\newcommand{\ene}{\sf N}
\newcommand{\be}{\sf B}

\newcommand{\ze}{\sf Z}
\newcommand{\re}{\sf R}
\newcommand{\fe}{\sf F}
\newcommand{\se}{\sf S}

\newcommand{\qu}{\sf Q}
\newcommand{\ga}{\sf G}

\newcommand{\ja}{\sf J}
\newcommand{\ia}{\sf I}

\newcommand{\K}{\mathbb{K}}
\newcommand{\Ker}{\sf Ker}
\newcommand{\im}{\sf Im}
\newcommand{\id}{\sf Id}
\begin{document}

\centerline{\bf Isoclinism  of Algebras With Bracket}

\bigskip
\centerline{J. M. Casas$^{(1)}$ and S. Hadi Jafari$^{(2)}$}
 \bigskip

\centerline{$^{(1)}$Departamento de Matemática Aplicada I \& CITMAga, Universidade de Vigo}
\centerline{E. E. Forestal, 36005 Pontevedra, Spain.}
\centerline{e-mail: jmcasas@uvigo.es\ (ORCID: 0000-0002-6556-6131)}

\centerline{$^{(2)}$ Department of Mathematics,  Mashhad Branch, Islamic Azad University}
\centerline{Mashhad, Iran}
\centerline{e-mail:  s.hadi\underline{ }jafari@yahoo.com\ (ORCID: 0000-0002-1166-0173)}

\bigskip \bigskip
\noindent

\bigskip \bigskip

\noindent{\bf Abstract:}

We introduce the notion of isoclinism between central extensions in the category of algebras with bracket.
We provide several equivalent conditions under which algebras with bracket are isoclinic.
We also study the connection between isoclinism and the Schur multiplier of algebras with bracket.
It is shown that for finite-dimensional central extensions of algebras with bracket with same dimension, the notion of isoclinism and isomorphism are equivalent. Furthermore, we indicate that all stem covers of an algebra with bracket are isoclinic. \\

\noindent Key Words: Algebra with bracket; Isoclinic extensions; Schur multiplier; Stem cover

\bigskip

\noindent AMS Subject Class. (2020): 16E40; 16E99; 16W99

\section{Introduction}

Isoclinism is an equivalence relation on the class of all groups, which is weaker than isomorphism. The study of the isoclinism of groups dates back to 1940 where P. Hall \cite{Ha} defined it with this idea in mind that all abelian groups fall into one equivalence class, namely they are equivalent to the trivial group. Two groups are said to be isoclinic if and only if there exists an isomorphism between their central quotients which induces an isomorphism between their commutator subgroups. In the case of finite groups, the smallest members of each equivalence class are called stem groups. The notion of isoclinism play an essential role in classification problems. Much is known about properties of groups being invariant under isoclinism. For instance, it is proved that, restricting ourselves to finite groups, many properties such as nilpotency, solvability, Schur covering group and commutative degree, are invariant under isoclinism.

In 1994, K. Moneyhun \cite{Mo} introduced the corresponding notion in the category of Lie algebras. {Obviously, he partitions all Lie algebras into disjoint equivalence classes, called isoclinism families. He also proved that each isoclinism family includes at least one stem Lie algebra of minimal dimension, which is characterized by the property that its center is contained in its derived subalgebra. The result concludes that the concepts of isoclinism and isomorphism between Lie algebras of the same finite dimension are equivalent.}

The concept of isoclinism on the central extensions of Lie algebras, which is a generalization of the above mentioned work of Moneyhun, is introduced in \cite{MSR}, in which it is shown that under some conditions, the concepts of isoclinism and isomorphism between the central extensions of finite-dimensional Lie algebras are identical. Later, in \cite{BC} the authors studied the central extensions relative to the Liezation functor, so called Lie-central extensions. In concrete, they deal with the notion of isoclinism on Lie-central extensions of Leibniz algebras.

The goal of the present article is the introduction of the notion of isoclinism on the central extensions of algebras with bracket structure, as a new sort of algebras, which was first introduced in \cite{CP}, with the novelty that it is an algebraic structure with two operations and, as far as we know, the study of the isoclinism relationship in algebraic structures with two operations has only been addressed in \cite{PU}. Roughly speaking, the algebras with bracket are associative (not necessarily commutative) algebras equipped with a bracket operation $[-,-]$ satisfying the relation
\[[a \cdot b,c]=[a,c]\cdot b+a\cdot[b,c]\]
with no further condition on $[-,-]$.
Conceptually, it is helpful to note that the algebras with bracket are a generalization of  Poisson algebras and studied as one of the lots of exotic algebraic structures have been developed by forgetting conditions on the classical structures.

The paper is structured as follows.
In the next section we fix our notations and establish some preliminary results.
In Section \ref{iso AWB}, we introduce the concept of isoclinism on the central extensions of algebras with bracket, and study the basic properties, while Section \ref{Schur} deals with the connection of isoclinism and Schur multiplier of algebras with bracket. Section \ref{Stem extensions} is devoted to  study stem extensions and stem covers of algebras with bracket, and showing that the concepts of isoclinism and isomorphism between  central extensions of finite-dimensional algebras with bracket are identical.

\section{ Algebras With Bracket} \label{AWB}
Throughout the paper we  fix a ground  field $\mathbb{K}$. All vector spaces are taken over $\mathbb{K}$. In what follows $\otimes$ means $\otimes_{\mathbb{K}}$.

 In this section we  summarize some basic notions about algebras with bracket,  the homology with trivial coefficients given in \cite{Ca, CP} and some results we will need later.

 \subsection{Basic definitions} \label{Basic}

 \begin{definition} \cite{CP} An  algebra with bracket (or \textup{AWB} for short) is  an associative $($not necessarily commutative$)$ algebra {\sf A}
equipped with a bilinear map $[-,-] : {\sf A} \times {\sf A} \to {\sf A}$, $(a, b)\mapsto [a,b]$ satisfying the following identity:
\begin{equation}\label{awbequa}
 [a  b,c] = [a,c]  b + a  [b,c]
\end{equation}
for all $a, b, c \in {\sf A}$.
\end{definition}

We denote by ${\sf AWB}$ the category whose objects are \textup{AWB}'s and whose morphisms are $\mathbb{K}$-linear maps
preserving the product and the bracket operation. It is a routine task to check that {\sf AWB} is a semi-abelian category \cite{BB, Mi}, so classical lemmas as Five Lemma holds for \textup{AWB}'s.

\begin{example}\label{ejemplos}\
\begin{enumerate}
\item[(a)] Poisson algebras, Noncommutative Leibniz-Poisson algebras \cite{CD} and Left-right noncommutative Poisson algebras \cite{CDL} are \textup{AWB}'s.

\item[(b)] Let {\sf A} be an associative $\K$-algebra equipped with a linear map $D: {\sf A} \to {\sf A}$. Then {\sf A} is an \textup{AWB} with
respect to the bracket $[a,b]:=aD(b) - D(b)a$. When $D = \id$, then {\sf A} is  an \textup{AWB} with respect to the usual bracket
for associative algebras, which is called the tautological \textup{AWB} associated to an associative algebra {\sf A}.

\item[(c)] If $({\sf A}, \prec, \succ)$ is a dendriform algebra (see \cite{Lo}), then $({\sf A},\star, [-,-])$ is an \textup{AWB}, where $a \star b = a \prec b + a \succ b$ and  $[a,b] = a \star b - b \star a$.

 \item[(d)] For examples coming from Physics we refer to \cite{Ka}.

 \item[(e)] Other examples, including 2-dimensional \textup{AWB}'s  can be found in \cite{Ca1}.

 \item[(f)] Let {\as} and {\be} be \textup{AWB}'s. Then $\as \times \be$ is an \textup{AWB} with respect to the componentwise operations.

\end{enumerate}
\end{example}

 The following notions for  an \textup{AWB} {\sf A} are given in \cite{Ca} and they agree with the corresponding  notions in semi-abelian categories. A {\it subalgebra} {\sf I} of {\sf A} is a $\mathbb{K}$-subspace which is closed under the product  and the bracket operation, that is, ${\sf I}~{\sf I} \subseteq {\sf I}$ and $[{\sf I},{\sf I}] \subseteq {\sf I}$.
{\sf I} is said to be a  {\it right (respectively, left) ideal} if  ${\sf A}~{\sf I} \subseteq {\sf I}$, $[{\sf A},{\sf I}] \subseteq {\sf I}$ (respectively, ${\sf I}~{\sf A} \subseteq {\sf I}$ , $[{\sf I},{\sf A}] \subseteq {\sf I}$).
If  ${\sf I}$ is both left and right ideal, then it is said to be  a {\it two-sided ideal}. In this case, the quotient
${\sf A}/{\sf I}$ is endowed with an \textup{AWB} structure  naturally induced from the operations on {\sf A}.

Let ${\sf I}, {\sf J}$ be two-sided ideals of {\sf A}. The {\it
commutator ideal} of {\sf I} and {\sf J} is the two-sided ideal  $[[{\sf I},{\sf J}]] = \langle \{ij, ji,[i,j],[j,i] \mid i \in {\sf I}, j \in {\sf J}
\} \rangle$ of {\sf I} and {\sf J}.
Obviously $[[{\sf I}, {\sf J}]] \subseteq {\sf I} \bigcap {\sf J}$. Observe that $[[{\sf I},{\sf J}]]$ is not a two-sided ideal of
{\sf A}, except when ${\sf I} = {\sf A}$ or ${\sf J} = {\sf A}$. In the particular case $ {\sf I} = {\sf J} = {\sf A}$, one obtains the definition of {\it derived algebra} of {\sf A}, i. e.,
$[[{\sf A},{\sf A}]]=\langle \{ab, [a,b] \mid a, b \in {\sf A} \} \rangle.$

The {\it center} of an \textup{AWB} {\sf A} is the two-sided ideal ${\ze}({\sf A}) =  \{a \in {\sf A} \mid ab
= 0 = ba,[a,b] = 0 = [b,a], {\rm for \ all}\ b \in {\sf A}\}.$
This object coincides with the categorical notion of annihilator in \cite{BoBu}, which is the centralizer object in the category, so it is usually also denoted by ${\sf Ann(A)}$.

An {\it abelian} algebra with bracket {\sf A} is an algebra with bracket with trivial product and bracket operation, i.e., ${\sf A} ~ {\sf
A} = 0 = [{\sf A},{\sf A}]$. Hence an algebra with bracket {\sf A} is abelian if and only if ${\sf A} = {\sf Z}({\sf A})$.


\subsection{Homology}

Let $V$ be a $\K$-vector space. Let be {${\cal R}^1(V)=V$ and ${\cal R}^n(V)=V^{\otimes n}\oplus V^{\otimes n},$ if $n\geq 2$. In order to distinguish elements from these tensor powers, we let $a_1\otimes \cdots \otimes  a_n$ be a typical element from the first component  of
${\cal R}^n(V)$, while $a_1\circ \cdots \circ a_n$ from the second component of ${\cal R}^n(V)$.

Let {\as} be an \textup{AWB} and consider  the complex $C_n^{AWB}({\as}):= R_{n+1}({\as}), n \geq 0$ with boundary maps given by
$$d_{n-1}(a_1\otimes \dots \otimes a_n) = \sum_{i=1}^{n-1} (-1)^{i+1}
a_1 \otimes \dots \otimes a_i  a_{i+1} \otimes \dots \otimes a_n$$
$$d_{n-1}(a_1\circ \cdots \circ a_n) = \sum_{i=1}^{n-1}
a_1 \otimes \dots \otimes  [a_i , a_n] \otimes  \dots \otimes
a_{n-1} + \sum_{i=1}^{n-2} (-1)^i a_1\circ \cdots \circ a_i a_{i+1} \circ \dots \circ a_n$$

The homology of the complex $(C_n^{AWB}({\as}),d_{n-1})$ is called the {\it homology with trivial
coefficients} of the \textup{AWB} {\as} and is denoted  by  ${\sf H_n^{AWB}}({\as})$ (see \cite{CP}).

Easy computations show that ${\sf H_0^{\AWB}}({\as}) \cong {\as}/[[{\as},{\as}]].$
On the other hand, for a free presentation $0 \to {\re} \to {\fe} \to {\as} \to 0$  of the
\textup{AWB} {\as}, by \cite[Corollary 2.14]{Ca},  the following  isomorphism holds:
\begin{equation} \label{Hopf}
{\sf H^{\AWB}_1}({\as}) \cong ({\re} \cap [[{\fe},{\fe}]])/[[{\re},{\fe}]].
\end{equation}

An extension of \textup{AWB}'s is a  short exact sequence of \textup{AWB}'s $(G) \colon 0 \to {\ene} \stackrel{\chi} \to {\ga} \overset{\pi} \to {\qu} \to 0$. It is said to be {\it central} if $[[{\ene}, {\ga}]] =0$ (equivalently, ${\ene} \subseteq {\ze}({\ga})$). A homomorphism between two extensions of \textup{AWB}'s  $(G_1)$ and  $(G_2)$ is a triple of \textup{AWB} homomorphisms $(\alpha, \beta, \gamma)$ that makes commutative the following diagram:
\[ \xymatrix{
(G_1) \colon  0 \ar[r] & {\ene}_1 \ar[r]^{\chi_1} \ar[d]^{\alpha} & {\ga}_1 \ar[r]^{\pi_1} \ar[d]^{\beta} &{\qu}_1 \ar[r] \ar[d]^{\gamma} & 0\\
(G_2) \colon 0 \ar[r] & {\ene}_2 \ar[r]^{\chi_2} & {\ga}_2 \ar[r]^{\pi_2} &{\qu}_2 \ar[r] & 0
} \]
When $\alpha, \beta, \gamma$ are bijective \textup{AWB} homomorphisms, $(\alpha, \beta, \gamma) \colon (G_1) \to (G_2)$ is called an isomorphism.
Due to \cite[Theorem 2.13]{Ca},  associated to an extension of \textup{AWB}'s $(G) : 0 \to \ene \stackrel{\chi} \to \ga \stackrel{\pi} \to \qu \to 0$ there exists the following exact and natural five-term exact sequence
\begin{equation}\label{five-term}
{\sf H_1^{\AWB}}({\ga}) \longrightarrow {\sf H_1^{\AWB}}({\qu}) \stackrel{\theta(G)}\longrightarrow \frac{\ene}{[[\ene, \ga]]} \longrightarrow {\sf H_0^{\AWB}}({\ga}) \longrightarrow {\sf H_0^{\AWB}}({\qu}) \longrightarrow 0.
 \end{equation}


\section{Isoclinic  \textup{AWB}'s} \label{iso AWB}

In this section we introduce the notion of isoclinism for \textup{AWB}'s and study its essential properties we will use in the following sections.

Consider the central extensions of \textup{AWB}'s $(G) : 0 \to \ene \stackrel{\chi} \to \ga \stackrel{\pi} \to \qu \to 0$ and $(G_i) : 0 \to {\ene}_i \stackrel{\chi_i}\to {\ga}_i \stackrel{\pi_i} \to {\qu}_i \to 0, i=1, 2,$

Let be $C : \qu \times \qu \to [[\ga, \ga]]$ given by $C(q_1,q_2)=[g_1,g_2]$ and $P : \qu \times \qu \to [[\ga, \ga]]$ given by $P(q_1,q_2)=g_1g_2$, where $\pi(g_j)=q_j, j= 1, 2$, the commutator maps associated to the extension $(G)$. In a similar way are defined the commutator maps $C_i, P_i$ corresponding to the extensions $(G_i), i = 1, 2$.

\begin{definition} \label{isoclinic}
The central extensions $(G_1)$ and $(G_2)$ are said to be isoclinic when there exist isomorphisms $\eta : \qu_1 \to \qu_2$ and $\xi : [[\ga_1, \ga_1]] \to [[\ga_2, \ga_2]]$ such that the following diagram is commutative:
\begin{equation}  \label{square isoclinic}
\xymatrix{
[[\ga_1, \ga_1]] \ar[d]^{\xi} & \qu_1 \times \qu_1 \ar[r]^{C_1} \ar[d]_{\eta \times \eta} \ar[l]_{P_1} & [[\ga_1, \ga_1]] \ar[d]^{\xi}\\
[[\ga_2, \ga_2]] & \qu_2 \times \qu_2 \ar[r]^{C_2} \ar[l]_{P_2} & [[\ga_2, \ga_2]]
}
\end{equation}

The pair $(\eta, \xi)$ is called an isoclinism from $(G_1)$ to $(G_2)$ and will be denoted by $(\eta, \xi) : (G_1) \to (G_2)$.
\end{definition}

Let $\qu$ be an \textup{AWB}, then we can construct the following central extension
\begin{equation} \label{Lie central extension}
(e_Q) : 0 \to \ze(\qu) \to \qu \stackrel{pr_{\qu}} \to \qu / \ze(\qu) \to 0.
\end{equation}

\begin{definition}
Two \textup{AWB}'s $\ga$ and $\qu$ are said to be isoclinic when $(e_G)$ and $(e_Q)$ are isoclinic central extensions.

An isoclinism $(\eta, \xi)$ from $(e_G)$ to $(e_Q)$ is also called an isoclinism  from $\ga$ to $\qu$, denoted by $(\eta, \xi) : \ga \sim \qu$.
\end{definition}

\begin{proposition} \label{isoclinism}
For an isoclinism $(\eta, \xi) : (G_1) \sim (G_2)$, the following statements hold:
\begin{enumerate}
\item[a)] $\eta$ induces an isomorphism $\eta' : \ga_1/\ze(\ga_1) \to \ga_2/\ze(\ga_2)$, and $(\eta', \xi)$ is an isoclinism from $\ga_1$ to $\ga_2$.
\item[b)] $\chi_1(\ene_1) = \ze(\ga_1)$ if and only if $\chi_2(\ene_2) = \ze(\ga_2)$.
\end{enumerate}
\end{proposition}
\begin{proof}
{\it a)} Clearly $\eta$ induces an epimorphism $\eta' : {\ga}_1/\ze({\ga}_1) \to {\ga}_2/\ze({\ga}_2)$, given by $\eta'(g_1+{\ze}({\ga}_1))=g_2+{\ze}({\ga}_2),$ such that $\pi_2(g_2)=\eta(\pi_1(g_1)), g_i \in {\ga}_i, i = 1, 2$.  So it is sufficient to show that $\eta(\ze(\ga_1)/N_1)=\ze(\ga_2)/N_2$.

Indeed, let $g_1 + {\ene}_1 \in {\ze}({\ga}_1)/{\ene}_1$ and $y_2 + {\ene}_2 \in {\ga}_2/{\ene}_2$. Then
\[[\eta(g_1 + {\ene}_1 ), y_2 + {\ene}_2] = [g_2 + {\ene}_2, y_2 + {\ene}_2] = [g_2, y_2]+{\ene}_2 = 0,\]
\noindent since $\pi_2[g_2,y_2] = [\eta(\pi_1(g_1)), \eta(\pi_1(y_1))] = \eta(\pi_1([g_1, y_1])) = 0$, for some $y_1 \in \ga_1$. Also
\[\eta(g_1 + {\ene}_1 )(y_2 + {\ene}_2) = (g_2 + {\ene}_2)(y_2 + {\ene}_2) = g_2 y_2+{\ene}_2 =0,\]
\noindent since $\pi_2(g_2y_2) = \eta(\pi_1(g_1))\eta(\pi_1(y_1)) = \eta(\pi_1(g_1y_1)) = 0$, for some $y_1 \in \ga_1$.

Similarly with the operations in the right side. Therefore $\eta(\ze(\ga_1)/N_1)\subseteq\ze(\ga_2)/N_2$. A similar argument for $\eta^{-1}$ yields the reverse inclusion which gives the assertion.

To show {\it b)}  it is enough to follow the same arguments as those of the proof of Proposition 3.4 {\it b)} in \cite{BC}.
\end{proof}

\begin{proposition} \label{equivalence relation}
Isoclinism of algebras with bracket is an equivalence relation.
\end{proposition}

Evidently, isoclinism between  central extensions divides the class of all central extensions of \textup{AWB}'s into equivalence classes, called isoclinism families. From Proposition \ref{isoclinism}, the classes of isoclinic \textup{AWB}'s can be regarded as those isoclinism classes of central extensions with kernel equal to {\ze}({\ga}), that is, the central extensions $(G)$ and $(e_G)$ are isomorphic.

\begin{proposition} \label{backward}
Let $(\eta, \xi) : (G_1) \sim (G_2)$ be an isoclinism. Then the backward induced extension $(\eta^{\ast}(G_2)): 0 \to \ene_2 \to \ga_2^{\eta} \stackrel{\overline{\pi}_2}\to \qu_1 \to 0$, obtained by pulling back along $\eta$ (where ${\ga_2}^{\eta} = \{ (g,q) \in  {\ga}_2 \times  {\qu}_1 \mid \pi_2(g) = \eta(q) \}$) is a central extension isomorphic to $(G_2)$, and $(\id_{\qu_1}, \xi) : (G_1) \sim (\eta^{\ast}(G_2))$ is an isoclinism.
\end{proposition}
\begin{proof}
The first statement only requires a routine checking.

To prove the next assertion, observe that $\xi$ induces an isomorphism (call it also $\xi$) $\xi : [[\ga_1, \ga_1]] \to [[\ga_2^{\eta}, \ga_2^{\eta}]]$ defined by $\xi([g_1,g_1'])=[(g_2,q_1),(g_2',q_1')]$ where $\pi_2(g_2) = \eta(q_1)=\eta (\pi_1(g_1))$ and $\pi_2(g_2') = \eta(q_1')=\eta (\pi_1(g_1'))$. Similarly we define $\xi(g_1g_1')=(g_2,q_1)(g_2',q_1')$. The commutativity of the diagram
\begin{equation}
\xymatrix{
[[\ga_1, \ga_1]] \ar[d]^{\xi} & \qu_1 \times \qu_1 \ar[r]^{C_1} \ar[d]_{\id_{\qu_1} \times \id_{\qu_1}} \ar[l]_{P_1} & [[\ga_1, \ga_1]] \ar[d]^{\xi}\\
[[\ga_2^{\eta}, \ga_2^{\eta}]] & \qu_1 \times \qu_1 \ar[r]^{C_2^{\eta}} \ar[l]_{P_2^{\eta}} & [[\ga_2^{\eta}, \ga_2^{\eta}]]
}
\end{equation}
where $C_2^{\eta}(q_1,q_1')=[(g_2,q_1),(g_2',q_1')]$ and $P_2^{\eta}(q_1q_1')=(g_2,q_1)(g_2',q_1')$, follows directly.
\end{proof}

\begin{proposition} \label{3.6}
Let $(\eta, \xi) : (G_1)  \sim (G_2)$ be an isoclinism. Then the following statements hold:
\begin{enumerate}
\item[a)] $\pi_2 (  \xi(g)) = \eta ( \pi_1(g))$, for all $g \in [[\ga_1, \ga_1]]$.
\item[b)] $\xi \left( \chi_1(\ene_1) \cap [[\ga_1, \ga_1]] \right) = \chi_2(\ene_2) \cap [[\ga_2, \ga_2]]$.
\item[c)] $\xi(g x) = h \xi(x); \xi(x g) = \xi(x) h; \xi([g,x]) = [h,\xi(x)]; \xi([x,g]) = [\xi(x), h]$, for all $x \in  [[{\ga}_1, {\ga}_1]], g \in {\ga}_1, h \in {\ga}_2$ such that $\eta (\pi_1(g)) = \pi_2(h)$.
\end{enumerate}
\end{proposition}
\begin{proof} The proofs of {\it a)}  and  {\it b)} follow the same arguments as the proof of Proposition 3.8 {\it a)}  and {\it b)} in \cite{BC}.

{\it c)} Let $x \in  [[{\ga}_1, {\ga}_1]], g \in {\ga}_1, h \in {\ga}_2$ such that $\eta(\pi_1(g)) = \pi_2(h)$. Then
\begin{align*}
\xi(g x)&=\xi \left(P_1 \left( \pi_1(g),\pi_1(x) \right) \right)\\
&=P_2(\eta ( \pi_1(g)),\eta ( \pi_1(x)))\\
&=P_2(\pi_2(h),\pi_2(\xi(x)))\\
&=h \xi(x),
\end{align*}
and
\begin{align*}
\xi([g,x])&=\xi(C_1(\pi_1(g),\pi_1(x)))\\
&=C_2(\eta ( \pi_1(g)),\eta ( \pi_1(x)))\\
&=C_2(\pi_2(h),\pi_2(\xi(x)))\\
&=[h,\xi(x)].
\end{align*}
The other equalities follow analogously.
\end{proof}

\begin{definition} \label{isoclinic1}
A homomorphism of  central extensions $(\alpha, \beta, \gamma) : (G_1) \to (G_2)$ is said to be isoclinic, if there exists an isomorphism $\beta' : [[\ga_1,\ga_1]] \to [[\ga_2,\ga_2]]$ with $(\gamma, \beta') : (G_1) \sim (G_2)$.

If $\beta$ is in addition an epimorphism (resp., monomorphism), then  $(\alpha, \beta, \gamma)$ is called an isoclinic epimorphism (resp., monomorphism).
\end{definition}

\begin{proposition} \label{equivalence}
For a homomorphism of central extensions $(\alpha, \beta, \gamma) : (G_1) \to (G_2)$, the following statements hold:
\begin{enumerate}
\item[a)]  $(\alpha, \beta, \gamma)$ is  isoclinic if and only if $\gamma$ is an isomorphism and $\Ker(\beta) \cap[[\ga_1,\ga_1]]  =0$.
    \item[b)] If  $(\alpha, \beta, \gamma)$ is isoclinic and $\beta'$ as in Definition \ref{isoclinic1}, then $\beta' = \beta_{\mid  [[\ga_1,\ga_1]] }$.
\end{enumerate}
\end{proposition}
\begin{proof}
Throughout this proof  and to simplify the calculations, we only consider elements of $[[{\ga}, {\ga}]]$ of the form $m=[g_1,h_1]+a_1b_1, g_1, h_1, a_1, b_1 \in {\ga}$, since a generic element of the commutator is of the form  $m=\displaystyle \sum_{i=1}^n \left( \lambda_i [g_i,h_i]+ \mu_ia_i b_i \right),$ $ g_{i}, h_{i}, a_i, b_i \in {\ga}, \lambda_i, \mu_i \in \mathbb{K}$, and the properties to prove  obviously extend by linearity.

{\it a)} Assume that $(\alpha, \beta, \gamma) : (G_1) \to (G_2)$ is isoclinic, then $(\gamma, \beta') : (G_1) \sim (G_2)$ is an isoclinism for some isomorphism $\beta' : [[\ga_1,\ga_1]] \to [[\ga_2,\ga_2]]$. This implies by definition that $\gamma$ is an isomorphism. Now let $m\in \Ker(\beta) \cap[[\ga_1,\ga_1]]$. Then $\beta(m)=0$ and $m=[g_1,h_1]+a_1b_1$ for some $g_1, h_1, a_1, b_1 \in {\ga}_1$. Also let $q_1=\pi_1(g_1)$, $q_2=\pi_1(h_1)$, $r_1=\pi_1(a_1)$, and $r_2=\pi_1(b_1)$. Since $(\gamma, \beta')$ is an isoclinism, we have
\begin{align*}
\beta'(m)&=\beta'([g_1,h_1]+a_1b_1)\\
&=\beta'(C_1(\pi_1(g_1),\pi_1(h_1))+{P_1(\pi_1(a_1),\pi_1(b_1))})\\
&=(\beta'\circ C_1)(\pi_1(g_1),\pi_1(h_1))+{(\beta'\circ P_1)(\pi_1(a_1),\pi_1(b_1))}\\
&=(C_2\circ (\gamma\times\gamma))(\pi_1(g_1),\pi_1(h_1))+{(P_2\circ (\gamma\times\gamma))(\pi_1(a_1),\pi_1(b_1))}\\
&=C_2(\gamma(\pi_1(g_1)),\gamma(\pi_1(h_1)))+{P_2(\gamma(\pi_1(a_1)),\gamma(\pi_1(b_1)))}\\
&=C_2(\pi_2(\beta(g_1)),\pi_2(\beta(h_1)))+{P_2(\pi_2(\beta(a_1)),\pi_2(\beta(b_1)))}\\
&=[\beta(g_1),\beta(h_1)]+{\beta(a_1)\beta(b_1)}\\
&=\beta([g_1,h_1]+{a_1b_1})\\
&=\beta(m)=0
\end{align*}
Since $\beta'$ is one-to-one, it follows that $m = 0$.

Conversely, assume that $\Ker(\beta) \cap[[\ga_1,\ga_1]]  =0$. Define $\beta' : [[\ga_1,\ga_1]] \to [[\ga_2,\ga_2]]$ by $\beta'(g)=\beta(g)$, which is one-to-one. It remains to show that $\beta'$ is onto. Let $y\in [[\ga_2,\ga_2]]$. Then $y=[g_2,h_2]+{a_2b_2}$ for some $g_2,h_2,a_2,b_2\in {\ga}_2$. Since $\pi_1$ and $\gamma$ are onto, it follows that $\pi_2(g_2)=(\gamma\circ \pi_1)(g_1)$, $\pi_2(h_2)=(\gamma\circ \pi_1)(h_1)$, $\pi_2(a_2)=(\gamma\circ \pi_1)(a_1)$, and $\pi_2(b_2)=(\gamma\circ \pi_1)(b_1)$ for some $g_1,h_1,a_1,b_1\in {\ga}_1$. By the homomorphism $(\alpha, \beta, \gamma)$, we have $(\gamma\circ \pi_1)(g_1)=(\pi_2\circ \beta)(g_1)$, $(\gamma\circ \pi_1)(h_1)=(\pi_2\circ \beta)(h_1)$, $(\gamma\circ \pi_1)(a_1)=(\pi_2\circ \beta)(a_1)$, and $(\gamma\circ \pi_1)(b_1)=(\pi_2\circ \beta)(b_1)$, which implies that $g_2-\beta(g_1)=\chi_2(n_2)$, $h_2-\beta(h_1)=\chi_2(m_2)$, $a_2-\beta(a_1)=\chi_2(p_2)$, and $b_2-\beta(b_1)=\chi_2(q_2)$ for some $n_2,m_2,p_2,q_2\in \ene_2$. We now have
\begin{align*}
\beta'([g_1,h_1]+a_1b_1)&=\beta([g_1,h_1]+a_1b_1)\\
&=[\beta(g_1),\beta(h_1)]+\beta(a_1)\beta(b_1)\\
&=[g_2-\chi_2(n_2),h_2-\chi_2(m_2)]+(a_2-\chi_2(p_2))(b_2-\chi_2(q_2))\\
&=y-([\chi_2(n_2),h_2-\chi_2(m_2)] + {\chi_2(p_2)(b_2-\chi_2(q_2))})\\
& -([g_2,\chi_2(m_2)]+ {a_2\chi_2(q_2)})\\
&=y.
\end{align*}
Statement {\it b)} follows directly from the proof of {\it a)}.
\end{proof}

The proofs of the following three results are similar to the corresponding facts for Leibniz algebras relative to the Liezation functor in \cite{BC} and are omitted.

\begin{proposition} \label{character}
Let $\beta : \ga \to \he$ be a homomorphism of \textup{AWB}'s. Then $\beta$ induces an isoclinic homomorphism from $(e_{G})$ to $(e_{H})$ if and only if $\Ker(\beta) \cap [[\ga,\ga]]=0$ and $\im(\beta) + \ze(\he) = \he$.

In this case we call $\beta$ an isoclinic homomorphism.
\end{proposition}

\begin{proposition} \label{composition}
For a homomorphism of central extensions $(\alpha, \beta, \gamma) : (G_1) \to (G_2)$, the following statements hold:
\begin{enumerate}
\item[a)]  A homomorphism $(\alpha, \beta, \gamma) : (G_1) \to (G_2)$ is  isoclinic if and only if $\gamma$ is an isomorphism and $\beta : \ga_1 \to \ga_2$ is an isoclinic homomorphism of \textup{AWB}'s.
\item[b)] The composition of isoclinic homomorphisms is an isoclinic homomorphism.
\item[c)] Each isoclinic homomorphism is a composition of an isoclinic epimorphism and an isoclinic monomorphism.
\end{enumerate}
\end{proposition}

\begin{proposition} \label{extra}
Let $\ga$ be an \textup{AWB}.
\begin{enumerate}
\item[a)] If $\as$ is an \textup{AWB} with trivial commutator, then $\ga \sim \ga \times \as$.
\item[b)] If ${\sf I}$ is a two-sided ideal of $\ga$, then $\ga/{\sf I} \sim \ga/({\sf I}\cap [[\ga, \ga]])$. In particular, $\ga \sim \ga/{\sf I}$ if and only if ${\sf I}\cap [[\ga, \ga]]=0$.
\item[c)] If $\he$ is a subalgebra of $\ga$, then $\he \sim \he + \ze(\ga)$. Moreover, $\he \sim \ga$ if and only if $\he + \ze(\ga)=\ga$.
\item[d)] If $\ga_i \sim \ga'_i$, $i=1,2$, then $\ga_1\times \ga_2 \sim \ga'_1\times \ga'_2$.
\end{enumerate}
\end{proposition}

The following result states that any two isoclinic \textup{AWB}'s $\ga_1$ and $\ga_2$ have a common isoclinic ancestor $\ga'$, i.e. $\ga_1$ and $\ga_2$ can be realized as quotients of the \textup{AWB} $\ga'$, while $\ga'$, $\ga_1$ and $\ga_2$ are in the same isoclinism family.
Also they have a common isoclinic descendant $\ga''$, i.e. $\ga_1$ and $\ga_2$ can be realized as subalgebras
of the \textup{AWB} $\ga''$, where $\ga''$, $\ga_1$ and $\ga_2$ are isoclinic to each other.

\begin{proposition} \label{quotient}
The following statements are equivalent:
\begin{enumerate}
\item[a)] The central extensions $(G_1)$ and $(G_2)$ are isoclinic.
\item[b)] There exists a central extension $(G')$ together with isoclinic epimorphisms from $(G')$ onto $(G_1)$ and $(G_2)$. In particular $\ga'$ contains two-sided ideals $\ga'_1$, $\ga'_2$ such that
\[\ga_1 \cong \ga'/\ga'_1\sim \ga'\sim \ga'/\ga'_2\cong\ga_2.\]
\item[c)] There exists a central extension $(G'')$ together with isoclinic monomorphisms from $(G_1)$ and $(G_2)$ into $(G'')$. In particular $\ga''$ contains subalgebras ${\ga}''_1$, ${\ga}''_2$ such that
\[\ga_1 \cong G''_1\sim \ga''\sim G''_2\cong\ga_2.\]
\end{enumerate}
\end{proposition}
\begin{proof}
{\it a)} $\Leftrightarrow$ {\it b)} Let $(G_1) \sim (G_2)$ and $(\eta, \xi)$ be an isoclinism between $(G_1) \sim (G_2)$. Consider a subalgebra $\ga'$ of $\ga_1\times\ga_2$ given by
\[{\ga}' =\{(g_1,g_2), g_i\in {\ga}_i, i=1,2 \mid \eta(\pi_1(g_1))=\pi_2(g_2)\}\]
Evidently the following diagrams are commutative for $i=1,2$:
\begin{equation} \label{widetilde}
\xymatrix{
(G') : 0 \ar[r] & {\ene}_1 \times {\ene}_2 \ar[r]^{\lambda} \ar[d]^{\sigma_i} & {\ga}' \ar[r]^{\rho} \ar[d]_{\tau_i}  & {\qu} \ar[r] \ar[d]_{\gamma_i} & 0\\
(G_i) : 0 \ar[r] & {\ene}_i \ar[r]^{\chi_i} & {\ga}_i \ar[r]^{\pi_i}  & {\qu}_i \ar[r] & 0
}
\end{equation}
where $\qu=\qu_1$, $\lambda=\chi_1\times \chi_2$, $\rho(g_1,g_2)=\pi_1(g_1)$, $\sigma_i : \ene_1 \times \ene_2\rightarrow \ene_i$, and $\tau_i : \ga'\rightarrow \ga_i$ is the $i$-th projection ($i = 1, 2$). Also we set $\gamma_1=\id_{\qu}$ and $\gamma_2=\eta$.
An easy computation shows that $(G')$ is a central extension.

Clearly ${\Ker}(\tau_1)=\{(0, g_2), g_2\in {\ene}_2\}$ and ${\Ker}(\tau_2)=\{(g_1,0), g_1\in {\ene}_1\}$. Put $\ga'_i=\Ker(\tau_i)$, $i=1,2$. Then  $\ga_1 \cong \ga'/\ga'_1$ and $\ga_2 \cong \ga'/\ga'_2$. Using the definition of isoclinism and Proposition \ref{3.6}, we conclude that every generator of the subalgebra $[[\ga',\ga']]$ is a linear combination of terms of  the form $([g_1,h_1]+a_1b_1,\xi([g_1,h_1]+a_1b_1))$ in which $g_1,h_1,a_1,b_1\in \ga_1$. It is readily to see that $[[\ga',\ga']]\cap \Ker(\tau_i)=0$, $i=1,2$. Thus Proposition \ref{equivalence} {\it a)} implies that $(\sigma_i,\tau_i,\gamma_i) : (G') \sim (G_i)$, $i=1,2$, and the assertion $\ga_1 \cong \ga'/\ga'_1\sim \ga'\sim \ga'/\ga'_2\cong\ga_2$ follows by Proposition \ref{extra} {\it b)}.

For proving ``if" part, according to diagram (\ref{widetilde}), note that the isomorphism $\eta : \qu_1 \rightarrow \qu_2$ induces an isoclinism from $(G_1)$ to $(G_2)$ if and only if $(\sigma_i,\tau_i,\gamma_i)$, $i=1,2$, are isoclinic epimorphisms by the adaptation of \cite[Proposition 3.14]{BC} to the framework of \textup{AWB}'s.

{\it a)} $\Leftrightarrow$ {\it c)} First, let ${\sf V}={\ga}'/{\ga}_1'\times {\ga}'/[[{\ga}',{\ga}']]$ and ${\sf W}=\{((n_1,0)+{\ga}_1',(n_1,0)+[[{\ga}',{\ga}']]) \mid n_1\in \ene_1\}$. It is easily verified that ${\sf W}$ is a two-sided ideal of ${\sf V}$ such that ${\sf W} \cap [[{\sf V},{\sf V}]]=0$. We define two homomorphisms $\delta_i : {\ga}_i \rightarrow {\sf V/W}, i =1, 2,$  as follows.

First, let $g_1\in \ga_1$ and choose $g_2\in\ga_2$ such that $(g_1, g_2) \in \ga'$. Then set $\delta_1(g_1)=((g_1, g_2)+{\ga}_1',0+[[\ga',\ga']])+{\sf W}$. Secondly, let $g_2\in\ga_2$ and choose $g_1\in \ga_1$ such that $(g_1, g_2) \in {\ga}'$. Then set $\delta_2(g_2)=((g_1, g_2)+{\ga}_1',(g_1, g_2)+[[{\ga}',{\ga}']])+{\sf W}$. One could easily show that $\delta_1$ and $\delta_2$ are monomorphisms. Take $\ga''=V/W$. Evidently the following diagram is commutative:
\begin{equation} \label{widetilde2}
\xymatrix{
(G_1) : 0 \ar[r] & \ene_1 \ar[r]^{\chi_1} \ar[d]^{{\delta_1}{\mid_{\ene_1}}} & \ga_1  \ar[d]_{\delta_1} \ar[r]^{\pi_1}  & \qu_1 \ar[r] \ar[d]_{\id_{\qu_1}} & 0\\
(G'') : 0 \ar[r] & {\sf V}'/{\sf W} \ar[r] & {\ga}'' \ar[r] & \qu_1 \ar[r] & 0
}
\end{equation}
where ${\sf V}'=\left({\ene}_1\times {\ene}_2\right)/{\ga}_1' \times {\ga}'/[[{\ga}',{\ga}']]$. Similarly the diagram
\begin{equation} \label{widetilde3}
\xymatrix{
(G_2) : 0 \ar[r] & \ene_2 \ar[r]^{\chi_2} \ar[d]^{{\delta_2}{\mid_{\ene_2}}} & \ga_2  \ar[d]_{\delta_2} \ar[r]^{\pi_2}  & \qu_2 \ar[r] \ar[d]_{\eta^{-1}} & 0\\
(G'') : 0 \ar[r] & {\sf V}'/{\sf W} \ar[r] & {\ga}'' \ar[r] & \qu_1 \ar[r] & 0
}
\end{equation}
 is commutative. Thus Proposition \ref{equivalence} {\it a)} concludes the first assertion.

Finally, as we can show that
\[\delta_1(\ga_1)+Z({\ga}'') = {\ga}'' = \delta_2(\ga_2)+Z({\ga}''),\]
then, putting ${\ga}''_1=\delta_1(\ga_1)$ and $\ga''_2=\delta_2(\ga_2)$, the second assertion follows by Proposition \ref{extra} {\it c)}.
The proof of ``if" part, is a direct consequence of Proposition \ref{equivalence relation}.
\end{proof}

The following Corollary is an immediate consequence of  Proposition \ref{quotient}.
\begin{corollary} \label{equivalent2}
Let ${\ga}_1$ and  ${\ga}_2$ be \textup{AWB}'s. Then the following statements are equivalent:
\begin{enumerate}
\item[a)]  ${\ga}_1$ and  ${\ga}_2$ are isoclinic.
\item[b)] There exist an \textup{AWB} with trivial commutator ${\as}$, a subalgebra ${\he}$ of ${\ga}_1\times {\as}$ with ${\he}+{\sf Z}({\ga}_1\times {\as})={\ga}_1\times {\as}$, and a two-sided ideal ${\ene}$ of ${\he}$ with ${\ene} \cap [[{\he},{\he}]] = 0$, such that ${\he}/{\ene}$ is isomorphic to ${\ga}_2$.
\item[c)] There exist an \textup{AWB} with trivial commutator ${\be}$, a two-sided ideal ${\me}$ of ${\ga}_1 \times {\be}$ with ${\me} \cap [[{\ga}_1\times {\be},{\ga}_1\times {\be}]] = 0$, and a subalgebra ${\ia}$ of $({\ga}_1 \times {\be})/{\me}$, with ${\ia}+ {\sf Z}(({\ga}_1 \times {\be})/{\me})=({\ga}_1 \times {\be})/{\me}$, such that ${\ia}$ is isomorphic to ${\ga}_2$.
\end{enumerate}
\end{corollary}
\begin{proof}
{\it a)} $\Leftrightarrow$ {\it b)} Assume that ${\ga}_1 \sim {\ga}_2$. Take ${\as} = {\ga}'/[[{\ga}',{\ga}']]$, ${\he} = \{(g'+{\ga}_1',g'+[[{\ga}',{\ga}']])\mid g'\in {\ga}'\}$ as a subalgebra of ${\ga}'/{\ga}_1' \times {\ga}'/[[{\ga}',{\ga}']]\cong {\ga}_1 \times {\as}$. Suppose $(g'+\ga_1',h'+[[\ga',\ga']])$ is an arbitrary element of $\ga_1\times \as$. It is obvious that
\[(g'+{\ga}_1',h'+[[{\ga}',{\ga}']])=(g'+{\ga}_1',g'+[[{\ga}',{\ga}']])+({\ga}_1',(h'-g')+[[{\ga}',{\ga}']])\in H+Z({\ga}_1\times {\as}).\]
Since the reverse containment follows immediately, then $\he+Z(\ga_1\times \as)=\ga_1\times \as$. Clearly the map $\delta : {\ga}' \rightarrow {\he}$ given by $\delta(g')=(g'+{\ga}_1',g'+[[{\ga}',{\ga}']])$ is an isomorphism. So taking ${\ene} = {\ga}'_2$ in Proposition \ref{quotient} {\it b)} it follows that ${\he}/{\ene} \cong {\ga}_2$.

{\it a)} $\Leftrightarrow$ {\it c)} Assume that ${\ga}_1\sim {\ga}_2$. Take ${\be} = {\ga}'/[[{\ga}',{\ga}']]$, ${\me}= {\sf W}$, and ${\ia}={\ga}''_2$ in Proposition \ref{quotient} {\it c)}.
\end{proof}


\section{Isoclinism and the Schur multiplier of \textup{AWB}} \label{Schur}

In this section we analyze the connection between isoclinism and the  first homology with trivial coefficients of algebras with bracket, which we call  the Schur multiplier thanks to the isomorphism (\ref{Hopf}).

\begin{lemma} \label{A}
Let $(\alpha, \beta, \gamma):(G_1) \to (G_2)$ be a homomorphism of central extensions and assume that $\gamma : \qu_1 \to \qu_2$ is an isomorphism. Then the following statements are equivalent:
\begin{enumerate}
\item[a)] $(\alpha, \beta, \gamma)$ is an isoclinism.
\item[b)] ${\sf H_1^{\AWB}}(\gamma) \left( {\Ker}(\theta(G_1)) \right) = {\Ker}(\theta(G_2))$.
\end{enumerate}
 \end{lemma}
\begin{proof}
From the natural five-term exact sequence (\ref{five-term}) it easily follows that ${\im}(\theta(G_i))={\ene}_i \cap [[{\ga}_i, {\ga}_i]], i = 1, 2$. Thus $\theta(G_i)$ induces a homomorphism $\theta'(G_i) : {\sf H_1^{\AWB}}({\qu}_i) \rightarrow [[{\ga}_i, {\ga}_i]]$ such that the sequence
\begin{equation} \label{four-term}
0 \longrightarrow {\Ker}(\theta(G_i)) \longrightarrow {\sf H_1^{\AWB}}({\qu}_i) \stackrel{\theta'(G_i)} \longrightarrow [[{\ga}_i, {\ga}_i]] \stackrel{\pi'_i} \longrightarrow [[{\qu}_i, {\qu}_i]] \longrightarrow 0,
 \end{equation}
is exact, where $\pi'_i$ is induced by $\pi_i \colon {\ga}_i \twoheadrightarrow {\qu}_i, i = 1, 2$. The naturality of $\theta(G_i)$ with respect to homomorphisms of central extensions (see \cite[Theorem 5.9]{EVdL}) yields the naturality of the above sequence which induces the following commutative diagram:
\begin{equation*}
\xymatrix{
{\Ker}(\theta(G_1)) \ar[d]^{{\sf H_1^{\AWB}}(\gamma)\mid} \ar@{>->}[r] \ar[d]  & {\sf H_1^{\AWB}}(\qu_1) \ar[rrr]^{\theta'(G_1)} {\ar@{->>}^{\theta(G_1)} (35,5); (53,20)*+{\ene_1 \cap [[\ga_1, \ga_1]]}}  {\ar@{>->} (57,16); (70,5)} \ar[d]^{{\sf H_1^{\AWB}}(\gamma)} &&  & \ar@{->>}[r]^{\pi_1'} [[\ga_1, \ga_1]] \ar[d]^{\beta'} & [[\qu_1, \qu_1]] \ar[d]^{\gamma'}\\
{\Ker}(\theta(G_2)) \ar@{>->}[r] & {\sf H_1^{\AWB}}(\qu_2) \ar[rrr]^{\theta'(G_2)} {\ar@{->>}^{\theta(G_2)} (35,-19); (53,-34)*+{\ene_2 \cap [[\ga_2, \ga_2]]}} {\ar@{>->} (57,-30); (70,-20)} &  & &[[\ga_2, \ga_2]] \ar@{->>}[r]^{\pi_2'} & [[\qu_2, \qu_2]]
}
\end{equation*}
where  $\beta'=\beta\mid_{[[{\ga}_1, {\ga}_1]]}$, $\gamma'=\gamma\mid_{[[{\qu}_1, {\qu}_1]]}$. Since $\gamma$ is an isomorphism, then $\gamma'$ and ${\sf H_1^{\AWB}}(\gamma)$ are also isomorphisms, the restriction of ${\sf H_1^{\AWB}}(\gamma)$ to ${\Ker}(\theta(G_1))$ is a monomorphism and $\beta'$ is an epimorphism.

By the commutativity of the left hand square in the diagram, we have that ${\sf H_1^{\AWB}}(\gamma)({\Ker}(\theta(G_1)))\subseteq {\Ker}(\theta(G_2))$. Conversely, for any $y\in {\Ker}(\theta(G_2))$, there exists $x\in{\sf H_1^{\AWB}}(\qu_1)$ such that ${\sf H_1^{\AWB}}(\gamma)(x)=y$. Now $0=\theta'(G_2)(y)=\theta'(G_2)\circ{\sf H_1^{\AWB}}(\gamma)(x)=\beta'\circ\theta'(G_1)(x)$. Hence $x\in {\Ker}(\theta(G_1))$ whenever $\beta'$ is a monomorphism.

Consequently, statement {\it b)} holds if and only if $\beta'$ is a monomorphism if and only if $(\alpha, \beta, \gamma)$ is an isoclinism (having in mind Proposition \ref{equivalence}).
\end{proof}

\begin{lemma} \label{B}
Let $\eta : \qu_1 \to \qu_2$ be an isomorphism of \textup{AWB} and $\ga'$ be as in the proof of Proposition \ref{quotient}, then $${\Ker}(\theta(G')) = {\sf H_1^{\AWB}}(\eta)^{-1}({\Ker}(\theta(G_2))) \cap {\Ker}(\theta(G_1)).$$
\end{lemma}
\begin{proof}
From diagram (\ref{widetilde}) and the naturality of $\theta(G)$ we have that
\[\theta(G')=\theta(G_1)\times \theta(G_2)\circ {{\sf H_1^{\AWB}}(\gamma_2)}.\]
Since ${\sf H_1^{\AWB}}(\gamma_2)$ is an isomorphism, then ${\Ker}(\theta(G_2))\circ {\sf H_1^{\AWB}}(\gamma_2)={\sf H_1^{\AWB}}(\gamma_2)^{-1}({\Ker}(\theta(G_2)))$, hence the required equality follows.
\end{proof}

By using the Schur multiplier, the following result presents equivalent conditions for central extensions of \textup{AWB}'s to be isoclinic.

\begin{theorem} \label{main1}
Let $\eta : \qu_1 \to \qu_2$ be an isomorphism of \textup{AWB}. Then the following statements are equivalent:
\begin{enumerate}
\item[a)] $\eta$ induces an isoclinism from $(G_1)$ to $(G_2)$.
\item[b)] There exists an isomorphism $\beta' : [[\ga_1, \ga_1]] \to [[\ga_2, \ga_2]]$ with $\beta' \circ \theta'(G_1) = \theta'(G_2) \circ{\sf H_1^{\AWB}}(\eta)$.
\item[c)] ${\sf H_1^{\AWB}}(\eta)({\Ker}(\theta(G_1))) = {\Ker}(\theta(G_2))$.
\end{enumerate}
\end{theorem}
\begin{proof}
{\it a)} $\Rightarrow$ {\it b)} Let $(\eta, \beta') : (G_1) \sim (G_2)$ be an isoclinism. Then Proposition \ref{equivalence} together with the note given in the proof of Proposition \ref{quotient} ({\it a)} $\Leftrightarrow$ {\it b)}) imply that $\beta'=\tau_2'\circ\tau_1'^{-1}$, where $\tau_i'={\tau_i}_{\mid [[\ga',\ga']]}$, and the adaptation of the  proof of \cite[Proposition 3.14]{BC} to the framework of \textup{AWB}'s shows that both of them are isomorphisms.
Then the naturality of sequence (\ref{four-term}) applied to diagram (\ref{widetilde}) implies
\[\theta'(G_2) \circ{\sf H_1^{\AWB}}(\gamma_2)=\tau_2'\circ\theta'(G')=\tau_2'\circ\tau_1'^{-1}\circ\theta'(G_1)\circ{\sf H_1^{\AWB}}(\gamma_1),\]
hence the required equality.

{\it b)} $\Rightarrow$ {\it c)} Since $\beta' \circ \theta'(G_1) = \theta'(G_2) \circ{\sf H_1^{\AWB}}(\eta)$, then ${\sf H_1^{\AWB}}(\eta)({\Ker}(\theta(G_1))) \subseteq {\Ker}(\theta(G_2))$. The converse inclusion is followed thanks to be $\beta'$ an isomorphism.

{\it c)} $\Rightarrow$ {\it a)} Let $(G')$ be as above, then Lemma \ref{B} implies that
\[{\Ker}(\theta(G'))={\sf H_1^{\AWB}}(\eta)^{-1}\circ{\sf H_1^{\AWB}}(\eta)({\Ker}(\theta(G_1)))\cap {\Ker}(\theta(G_1))={\Ker}(\theta(G_1)).\]
From diagram (\ref{widetilde}), the epimorphisms $(\sigma_i,\tau_i,\gamma_i) : (G') \sim (G_i)$ and the above equality imply that
\[{\sf H_1^{\AWB}}(\gamma_i)({\Ker}(\theta(G')))={\Ker}(\theta(G_i)), i=1,2.\]
Lemma \ref{A} implies that $(\sigma_i,\tau_i,\gamma_i) : (G') \sim (G_i)$, $i=1,2$. Now the note in the proof of Proposition \ref{quotient} ({\it a)} $\Leftrightarrow$ {\it b)})
ends the proof.
\end{proof}

\begin{proposition} \label{4.4}
For any  \textup{AWB} {\qu}, there exists a homomorphism $\theta_{\qu} \colon {\sf H_1^{\AWB}}({\qu}/{\sf Z}({\qu}))$ $\to [[{\qu},{\qu}]]$ such that the following sequence is exact and natural:
\begin{equation} \label{four-term 1}
0 \to \frac{{\re} \cap [[{\fe},{\fe}]]}{[[{\se},{\fe}]]} \to {\sf H_1^{\AWB}}({\qu}/{\sf Z}({\qu})) \overset{\theta_{\qu}}\to [[{\qu},{\qu}]] \overset{\pi'}\to [[{\qu}/{\sf Z}({\qu}),{\qu}/{\sf Z}({\qu})]] \to 0
\end{equation}
where $\pi'$ denotes the restriction of the canonical projection $\pi \colon  {\qu} \to {\qu}/{\sf Z}({\qu})$ and $0 \to {\re} \to {\fe} \overset{\rho} \to {\qu} \to 0$ is a free presentation of {\qu}.
\end{proposition}
\begin{proof}
For a free presentation ${\fe}/{\re} \cong {\qu}$, suppose ${\sf Z}({\qu})\cong {\se}/{\re}$ for a suitable ideal $\se$ of $\fe$. Consider the following diagram of free presentations:
\[
\xymatrix{& & 0 \ar[d] & 0 \ar[ld]\\
&  & {\re} \ar[d]\ar[ld] \\
0 \ar[r] & {\se} \ \ar[r] \ar[d] & {\fe} \ar[d]^{\rho} \ar[rd]^{\pi \circ \rho} \\
0 \ar[r]& {\sf Z}({\qu}) \ar[r] \ar[d]& {\qu} \ar[r]^{\pi} \ar[d]& {\qu}/{\sf Z}({\qu}) \ar[r] \ar[dr] & 0 \\
 & 0 & 0 &  & 0
}
\]
Since $[[{\se},{\fe}]] \subseteq {\re}$, then $\theta_{\qu} \colon {\sf H_1^{\AWB}}({\qu}/{\sf Z}({\qu})) \cong \frac{{\se} \cap [[{\fe},{\fe}]]}{[[{\se},{\fe}]]} \to [[{\qu},{\qu}]], \theta_{\qu}(x+[[{\se},{\fe}]]) = x +{\re}$, is a well-defined homomorphism such that $\im(\theta_{\qu}) = [[{\qu},{\qu}]] \cap {\sf Z}({\qu})$ and $\Ker(\theta_{\qu}) = \frac{{\re} \cap [[{\fe},{\fe}]]}{[[{\se},{\fe}]]}$.
\end{proof}


\section{Stem extensions of \textup{AWB}'s} \label{Stem extensions}

This section is dedicated to analyzing the interconnections between isoclinism, stem extensions and stem covers.

\begin{definition}
A central extension $(G) : 0 \to {\ene} \stackrel{\chi} \to {\ga} \stackrel{\pi} \to {\qu} \to 0$ is said to be:
\begin{enumerate}
\item[a)] a stem extension (of $\qu$) if ${\ga}/[[{\ga},{\ga}]]\cong {\qu}/[[{\qu},{\qu}]]$.
\item[b)] a stem cover (of $\qu$) if ${\ga}/[[{\ga},{\ga}]]\cong {\qu}/[[{\qu},{\qu}]]$ and the induced map ${\sf H}_1^{\AWB}({\ga})\rightarrow {\sf H}_1^{\AWB}({\qu})$ is the zero map.
\end{enumerate}
\end{definition}
In case {\it b)}, $\ga$ is said to be a cover (or covering \textup{AWB}) of $\qu$.
\begin{proposition} \label{stem}
For a central extension $(G) \colon 0 \to \ene \stackrel{\chi} \to \ga \stackrel{\pi} \to \qu \to 0$, the following statements are equivalent:
\begin{enumerate}
\item[a)] $(G)$ is a stem extension.
\item[b)] The induced map ${\ene} \rightarrow {\sf H}_0^{\AWB}({\qu})$ is the zero map.
\item[c)] $\theta(G) \colon {\sf H}_1^{\AWB}({\qu}) \rightarrow {\ene}$ is an epimorphism.
\item[d)] ${\ene} \subseteq [[{\ga},{\ga}]]$.
\end{enumerate}
\end{proposition}
\begin{proof}
The equivalences between {\it a), b)} and {\it c)} follow from the exact sequence (\ref{five-term}). The equivalence between {\it a)} and {\it d)} is a consequence of the isomorphisms ${\ga}/[[\ga,\ga]]\cong {\qu}/[[{\qu},{\qu}]]\cong {\ga}/([[{\ga},{\ga}]]+{\ene})$ which implies that $[[{\ga},{\ga}]]+{\ene}=[[{\ga},{\ga}]]$.
\end{proof}

\begin{corollary}
Let $(G_i) \colon 0 \to {\ene_i} \stackrel{\chi_i} \to {\ga_i} \stackrel{\pi_i} \to {\qu_i} \to 0$, $i=1,2$, be two isoclinic stem extensions,  then $\ene_1\cong \ene_2$. In particular, if the \textup{AWB}'s ${\ga_i}$, $i = 1, 2$, are finite-dimensional, then $\textup{dim}\,\ga_1 = \textup{dim}\,\ga_2$.
\end{corollary}
\begin{proof}
Straightforward consequence of Proposition \ref{3.6} {\it b)}.
\end{proof}

\begin{lemma} \label{5.3}
A central extension $(G) \colon 0 \to {\ene} \stackrel{\chi} \to {\ga} \stackrel{\pi} \to {\qu} \to 0$ is stem if and only if ${\me} \cap [[{\ga},{\ga}]] \neq 0$, for every non-zero two-sided ideal {\me} of {\ene}.
\end{lemma}
\begin{proof}
Assume $(G)$ is stem and {\me} is a two-sided ideal of {\ene} such that  ${\me} \cap [[{\ga},{\ga}]] = 0$. Then ${\me} = {\me} \cap {\ene} \subseteq {\me} \cap [[{\ga},{\ga}]] = 0$, consequently ${\me} = 0$.

Conversely, suppose on the contrary that ${\ene} \nsubseteq [[{\ga},{\ga}]]$. Let $ \langle x \rangle$ be a two-sided ideal of {\ga} spanned by $x \in {\ene} \backslash  [[{\ga},{\ga}]]$. If $y \in \langle x \rangle\cap [[{\ga}, {\ga}]]$, then $y = \lambda x \in {\ene}$ for some $\lambda \in \mathbb{K}$. On the other hand $y \in[[{\ga}, {\ga}]]$, but $x \notin [[{\ga}, {\ga}]]$, hence $\lambda x \notin [[{\ga}, {\ga}]]$, except $\lambda = 0$. Therefore $y = 0$ whence $\langle x \rangle \cap [[{\ga},{\ga}]] = 0$, which is a contradiction.
\end{proof}

In the following, we establish that each isoclinism class of central extensions contains at least one stem extension.

\begin{proposition} \label{stemexistence}
Any central extension is isoclinic to a stem extension.
\end{proposition}
\begin{proof}
Suppose that $(G) \colon 0 \to {\ene} \stackrel{\chi} \to {\ga} \stackrel{\pi} \to {\qu} \to 0$ is a central extension. Let  $\mathcal{A}= \{ {\ja} \mid {\ja} ~\text{is a  two-sided ideal of} ~{\ga} ~\text{such that} ~{\ja} \subseteq {\ene}  ~\text{and} ~ {\ja} \cap [[\ga,\ga]] = 0\}$. The set  is non-empty because it contains the zero ideal. We define a partial ordering on $\mathcal{A}$ by inclusion and evidently, by Zorn's lemma, we can find a maximal two-sided ideal $\me$ in $\mathcal{A}$. Since $\me \cap [[\ga,\ga]] = 0$, it follows by Proposition \ref{equivalence} that $(nat,nat,id_{\qu}) : (G) \to (G/M)$ is isoclinic, where
$(G/M)$ denotes the central extension $0 \to {\ene/\me} \stackrel{\overline{\chi}} \to {\ga/\me} \stackrel{\overline{\pi}} \to {\qu} \to 0$ and $nat$ denotes the corresponding projection.

We claim that $(G/M)$ is a stem extension. Indeed, assume that ${\ja}/{\me} \cap [[{\ga}/{\me}, {\ga}/{\me}]] = {\me}$, for some two-sided ideal {\ja} of {\ga} such that ${\me} \subseteq {\ja} \subseteq {\ene}$. If $x \in {\ja} \cap [[{\ga},{\ga}]]$, then $x+{\me} \in {\ja}/{\me} \cap [[{\ga}/{\me}, {\ga}/{\me}]] = {\me}$, therefore $x \in {\me}$ and hence $x \in {\me} \cap [[{\ga},{\ga}]] = 0$. In consequence, ${\ja} \cap [[{\ga},{\ga}]] = 0$ and ${\ja} \in \mathcal{A}$. Now the maximality of {\me} implies that ${\ja} = {\me}$ and Lemma \ref{5.3} concludes the proof.
\end{proof}

\begin{corollary}
Let ${\cal C}$ be an isoclinism family of finite-dimensional \textup{AWB}'s and ${\ga} \in {\cal C}$. Then ${\ga}$ is a stem \textup{AWB} (i.e. $\ze({\ga}) \subseteq [[{\ga}, {\ga}]]$) if and only if $\textup{dim}\,{\ga} = \textup{min}\,\{ \textup{dim}\,{\he} \mid {\he} \in {\cal C} \}$.
\end{corollary}
\begin{proof}
Assume that ${\ga}\sim {\he}$. Then
\begin{align*}
[[{\he}, {\he}]]/\left( [[{\he}, {\he}]] \right) \cap \ze({\he})&\cong ([[{\he}, {\he}]]+ \ze({\he}))/\ze({\he})\\
&=[[\he/\ze({\he}),\he/\ze({\he})]]\\
&\cong [[\ga/\ze({\ga}),\ga/\ze({\ga})]]\\
&=([[{\ga}, {\ga}]]+ \ze({\ga}))/\ze({\ga})\\
&\cong [[{\ga}, {\ga}]]/\ze({\ga})
\end{align*}
and $[[{\ga}, {\ga}]]\cong [[{\he}, {\he}]]$. So $\text{dim}~ \ze({\ga})=\text{dim}([[{\he}, {\he}]]\cap \ze({\he}))\leq \text{dim}~\ze({\he})$ which implies that $\text{dim}~ \ga \leq \text{dim}~\he$.

Now suppose that ${\ga}$ is a finite-dimensional \textup{AWB} such that dim  {\ga} = min $\{ \text{dim}~ {\he} \mid {\he} \in {\cal C} \}$. By Proposition \ref{stemexistence} there exists a two-sided ideal $\me$ of $\ga$, $\me\subseteq \ze({\ga})$  such that $\ga/\me$ is a stem \textup{AWB} and $\ga/\me\sim \ga$. It follows that $\me=0$ and $\ga$ is stem \textup{AWB}.
\end{proof}

Now  we prove that the notions of isoclinism and  isomorphism are identical in case of finite-dimensional stem extensions of \textup{AWB}'s.
Following \cite{CP}, given an abelian extension $(G) \colon 0 \to {\ene} \stackrel{\chi}{\to} {\ga} \stackrel{\pi}{\to} {\qu} \to 0$ of  \textup{AWB}'s, a {\it factor set} on $(G)$ is a pair $(f, g)$, where $f \colon {\qu}^{\otimes 2} \to {\ene}$ is a 2-cocycle in the Hochschild complex $C^*({\qu},{\ene}) = {\sf Hom}({\qu}^{\otimes *}, {\ene})$, while $g \colon {\qu} \to {\ene}^e = {\sf Hom}({\qu},{\ene})$ is a linear map, such that
\begin{equation}
[f(a,b),c]-f([a,c],b) - f(a,[b,c]) = (g(a)(c))b + a(g(b)(c)) - (g(ab))(c)
\end{equation}
for all $a, b, c\in {\qu}$.

If $(f, g)$ is a factor set on the extension $(G)$, then the  vector space ${\ene} \oplus {\qu}$ endowed with the operations
\begin{align*}
(n,q)(n',q') & = (qn'+nq'+f(q,q'), qq'),\\
[(n,q),(n',q')] & = ([q,n']+[n,q']+(g(q))(q'), [q,q']),
\end{align*}
has a structure of \textup{AWB}, where $qn', nq', [q,n'], [n,q']$ means the actions provided by the central extension $(G)$,
i.e.  $qn'= \partial(q) \chi(n'), [q,n']= [\partial(q), \chi(n')], nq'= \chi(n) \partial(q'), [n, q']= [\chi(n), \partial(q')]$, for all $q, q' \in {\qu}, n, n' \in {\ene}$, with $\partial$ a linear section of $\pi$,  which is denoted by ${\ene} \oplus_{(f, g)} {\qu}$ and it gives rise to the central extension $(G_{(f,g)} ) \colon 0 \to {\ene}_{(f, g)}  \stackrel{i}{\to} {\ene} \oplus_{(f, g)} {\qu}  \stackrel{p}{\to} {\qu} \to 0$ with $i(n)=(n,0)$ and $p(n,q)=q$. Obviously, ${\ene}_{(f, g)} = {\Ker}(p) \cong {\Ker}({\pi})$. Note that when the extension $(G)$ is central, then the above actions are trivial and the operations reduce to
\begin{align*}
(n,q)(n',q') & = (f(q,q'), qq'),\\
[(n,q),(n',q')] & = ((g(q))(q'), [q,q']).
\end{align*}

\begin{theorem} \label{4.1}
Let $(G_i) \colon 0 \to {\ene}_i \stackrel{\chi_i}\to  {\ga}_i \stackrel{\pi_i} \to {\qu}_i \to 0$, $i=1,2$, be two stem extensions of finite-dimensional \textup{AWB}'s. Then $(G_1)$ and $(G_2)$ are isoclinic if and only if they are \textup{AWB} isomorphic.
\end{theorem}
\begin{proof}
The ``if'' part  is a consequence of Proposition \ref{equivalence}. For the converse implication, let $(\eta,\xi) \colon (G_1) \sim (G_2)$ be an isoclinism and consider the factor set $(f, g)$ on $(G_1)$ given by
\begin{align*}
f(a, b) &=  \partial_1(a) \partial_1(b) - \partial_1(ab),\\
g(a)(b) &=  [\partial_1(a), \partial_1(b)] - \partial_1([a,b]),
\end{align*}
where $\partial_1$ is a linear section of $\pi_1$, and  $a, b \in {\qu_1}$.

The \textup{AWB} isomorphism  $\theta \colon {\ene}_1\oplus_{(f, g)} {\qu}_1 \rightarrow
{\ga}_1$ defined by $\theta(n, q) = n +\partial_1(q), n \in {\ene}_1, q \in {\qu}_1$, induces the following commutative diagram:
\[ \xymatrix{
(G_1)_{(f, g)} : 0 \ar[r] & {\ene}_{1_{(f,g)}} \ar[r]^{i_1~~~~~ } \ar@{=}[d] & {\ene}_1 \oplus_{(f, g)} {\qu}_1   \ar@{-->}[d]^{\theta} \ar[r]^{~~~~~p_1}  & {\qu}_1 \ar[r] \ar@{=}[d]& 0\\
(G_1) : 0 \ar[r] & {\ene}_1 \ar[r]^{\chi_1} & {\ga}_1 \ar[r]^{\pi_1} & {\qu}_1 \ar[r] \ar@/_1pc/[l]_{\partial_1} & 0
} \]
Consequently, we have $(G_1) \cong (G_1)_{(f, g)}$.

The isoclinism $(\eta, \xi)$ implies that $\xi({\ene}_1) = {\ene}_2$ by Proposition \ref{3.6} and keeping in mind that $(G_i)$ are stem extensions for $i =1, 2$.  With an argument similar to the previous one we also have a factor set $(h, k)$ provided by the linear section $\partial_2$ corresponding to the extension $(G_2)$ such that $(G_2) \cong (G_2)_{(h, k)}$. Now the maps $F \colon {\qu}_1 \times {\qu}_1 \to {\ene}_1, F(a, b) = \xi^{-1} \left(  h(\eta(a), \eta(b)) \right)$ and $G \colon {\qu}_1 \to {\sf Hom}({\qu}_1, {\ene}_1), G(a)(b) = \xi^{-1} \left( k(\eta(a))(\eta(b)) \right)$ define a factor set which gives rise to the central extension $(G_1)_{(F, G)}$.  Moreover, the following diagram is commutative:
\[ \xymatrix{
(G_1)_{(F, G)} : 0 \ar[r] & {\ene}_{1_{(F,G)}} \ar[r]^{I_1~~~~~ } \ar[d]^{\omega_{\mid}} & {\ene}_1 \oplus_{(F, G)} {\qu}_1   \ar@{-->}[d]^{\omega} \ar[r]^{~~~~~\prod_1}  & {\qu}_1 \ar[r] \ar[d]^{\eta}& 0\\
(G_2)_{(h, k)} : 0 \ar[r] & {\ene}_{2_{(h,k)}} \ar[r]^{i_2 ~~~~~} & {\ene}_2 \oplus_{(h, k)} {\qu}_2 \ar[r]^{~~~~~p_2} & {\qu}_2 \ar[r] \ar@/_1pc/[l]_{\partial_2} & 0
} \]
where $\omega(n_1, q_1) = (\xi(n_1), \eta(q_1)), n_1 \in {\ene}_1, q_1 \in {\qu}_1$. Moreover,  $(\omega_{\mid},\omega, \eta) \colon (G_1)_{(F, G)} \to (G_2)_{(h, k)}$ is an isomorphism, hence $(G_2) \cong (G_1)_{(F, G)}$.

To conclude the proof we will show that $(G_1)_{(f, g)} \cong (G_1)_{(F, G)}$. Indeed, by the previous arguments, we can assume that $(\eta', \xi') \colon (G_1)_{(f, g)} \sim (G_1)_{(F, G)}$ and $\xi'(N_{1_{(f, g)}}) = N_{1_{(F,G)}}$.

On the other hand, for all $q_1, q_2 \in {\qu}_1$ we have:
\[
\begin{array}{rcl}  \xi'(f(q_1, q_2), q_1 q_2) &=&  \xi'(f(q_1, q_2), 0) + \xi'(0, q_1 q_2)\\
 \xi'((0,q_1) (0, q_2)) &=& (0, \eta'(q_1)) (0, \eta'(q_2)) = (F(\eta'(q_1),\eta'(q_2)), \eta'(q_1 q_2))\\
 \\
  \xi'(g(q_1, q_2), [q_1, q_2]) &=&  \xi'(g(q_1, q_2), 0) + \xi'(0,[q_1, q_2])\\
 \xi'([(0,q_1),(0, q_2)]) &=& [(0, \eta'(q_1)), (0, \eta'(q_2))] = (G(\eta'(q_1))(\eta'(q_2)), \eta'([q_1, q_2]))
\end{array}
\]
If we take $d(q_1 q_2)$ as the first component of $\xi'(0, q_1 q_2)$ and  $d([q_1, q_2])$ as the first component of $\xi'(0, [q_1, q_2])$, then we obtain a linear map $d \colon [[{\qu}_1, {\qu}_1]]\to {\ene}_1$ satisfying
\[
\begin{array}{rcl}  \rho \circ \xi'(f(q_1, q_2), 0) + d(q_1 q_2) &=&  F(\eta'(q_1),\eta'(q_2))\\
 \rho \circ \xi'(g(q_1,  q_2), 0) + d([q_1, q_2]) &=&  G(\eta'(q_1),\eta'(q_2))
\end{array}
\]
where $\rho \colon {\ene}_{1_{(F,G)}} \to {\ene}_1$ is the canonical projection. Let {\sf T} be the complementary vector subspace to $[[{\qu}_1, {\qu}_1]]$ in ${\qu}_1$, i.e. ${\qu}_1 = [[{\qu}_1, {\qu}_1]] + {\sf T}$, then we can linearly extend $d$ to $\bar{d} \colon {\qu}_1 \to {\ene}_1$ by defining
\[
\bar{d} (x) = \left \{ \begin{array}{rcl} d(x) & \text{if}  & x \in [[{\qu}_1, {\qu}_1]] \\
0 & \text{if} & x \in {\sf T}
\end{array} \right.
\]
Finally, the linear map $\lambda \colon {\ene}_1 \oplus_{(f, g)} {\qu}_1 \to {\ene}_1 \oplus_{(F, G)} {\qu}_1$ given by $\lambda(n, q) = \xi'(n, 0 ) + (\bar{d}(q), \eta'(q)), n \in {\ene}_1, q \in {\qu}_1$ is an \textup{AWB} isomorphism, which is a consequence of the commutativity of the following diagram:
\[ \xymatrix{
 0 \ar[r] & {\ene}_{1_{(f,g)}} \ar[r]^{} \ar[d]^{\xi'} & {\ene}_1 \oplus_{(f, g)} {\qu}_1   \ar@{-->}[d]^{\lambda} \ar[r]^{}  & {\qu}_1 \ar[r] \ar[d]^{\eta'} & 0\\
 0 \ar[r] & {\ene}_{1_{(F,G)}} \ar[r]^{} & {\ene}_1 \oplus_{(F, G)} {\qu}_1 \ar[r]^{} & {\qu}_1 \ar[r]  & 0
} \]

\noindent Therefore $(\xi', \lambda, \eta') \colon (G_1)_{(f,g)}) \to (G_1)_{(F, G)}$ is an isomorphism, as required.
\end{proof}

Suppose that $(G) \colon 0 \to {\ene} \stackrel{\chi} \to {\ga} \stackrel{\pi} \to {\qu} \to 0$ is a central extension of \textup{AWB}'s and $\as$ is an abelian \textup{AWB}. Then the extension $$(G\oplus A) \colon 0 \to {\ene\oplus \as} \stackrel{(\chi, \id)} \to {\ga\oplus \as} \stackrel{\pi \circ \pi_{\ga}} \to {\qu} \to 0$$
is central, in which $\pi_{\ga} : \ga\oplus \as \rightarrow \ga$ denotes the projective homomorphism. It follows from Proposition \ref{equivalence} {\it a)} that the homomorphism $(\pi_{\ene},\pi_{\ga},{\id}_{\qu}) \colon (G\oplus A)$ $\to (G)$ is an isoclinism.

\begin{theorem} \label{4.3}
Let $\mathcal{C}$ be an isoclinism family of finite-dimensional central extensions of \textup{AWB}'s. Then $\mathcal{C}$ contains at least one stem extension $(H)$, and every central extension lying in $\mathcal{C}$ has the form $(H\oplus A)$, in which $\as$ is a finite-dimensional abelian \textup{AWB}.
\end{theorem}
\begin{proof}
By Proposition \ref{stemexistence}, $\mathcal{C}$ admits a stem extension $(H)$.
Evidently, for any finite-dimensional abelian \textup{AWB} $\as$, the extensions $(H\oplus A)$ and $(H)$ are isoclinic and consequently $(H\oplus A) \in \mathcal{C}$. Now, suppose $(G) \colon 0 \to {\ene} \stackrel{\chi} \to {\ga} \stackrel{\pi} \to {\qu} \to 0$ is an arbitrary central extension in $\mathcal{C}$. Following the proof of Proposition \ref{stemexistence}, we can find a two-sided ideal $\me$ of {\ga} with ${\me} \subseteq {\ene}$, such that $\ene=\me \oplus ({\ene} \cap [[\ga,\ga]])$ and $(G/M) \colon 0 \to {\ene}/{\me} \overset{\bar{\chi}} \to {\ga}/{\me} \overset{\bar{\pi}} \to {\qu} \to 0$ is a stem extension in ${\cal C}$. Consider the subspace ${\sf T}$ of $\ga$ such that ${\sf T}$ is the complement of $\me$ in $\ga$ and $[[\ga,\ga]]\subseteq {\sf T}$. Then {\sf T} is a two-sided ideal of $\ga$ with $\pi({\sf T}) = \qu$, and the stem extensions $(H)$ and $(G/M)\cong (G_T)$ are isoclinic, where $(G_T) \colon 0 \to {\sf \ene\cap T} \stackrel{\chi} \to {\sf T} \stackrel{\pi_{\mid  \sf T }} \to {\qu} \to 0$. According to Theorem \ref{4.1}, we have $(G_T)\cong (H)$ and hence $(G)\cong (G_T\oplus M)\cong (H\oplus M)$, in which $M$ is a finite-dimensional abelian \textup{AWB} and the result is proved.
\end{proof}

\begin{corollary} \label{4.2}
Let $(G_i) \colon 0 \to {\ene_i} \stackrel{\chi_i} \to {\ga_i} \stackrel{\pi_i} \to {\qu_i} \to 0$, $i=1,2$, be finite-dimensional central extensions with $\textup{dim}~ {\ga}_1=\textup{dim} ~{\ga}_2$. Then $(G_1)$ and $(G_2)$ are isoclinic if and only if they are isomorphic.
\end{corollary}

 At the end of this section, we establish some descriptions of stem covers for $\AWB$'s.

\begin{proposition} \label{stemcover}
For a central extension $(G) \colon 0 \to {\ene} \stackrel{\chi} \to {\ga} \stackrel{\pi} \to {\qu} \to 0$, the following statements are equivalent:
\begin{enumerate}
\item[a)] $(G)$ is a stem cover.
\item[b)] $\theta(G) \colon {\sf H}_1^{\AWB}({\qu}) \rightarrow \ene$ is an isomorphism.
\end{enumerate}
\end{proposition}
\begin{proof}
This is a direct consequence of the five-term exact sequence $(\ref{five-term})$.
\end{proof}

So an extension $(G) \colon 0 \to {\ene} \stackrel{\chi} \to {\ga} \stackrel{\pi} \to {\qu} \to 0$ of \textup{AWB}'s is  a stem cover of $\qu$ if $\ene\subseteq{\ze}(\ga)\cap [[\ga,\ga]]$ and $\ene\cong {\sf H}_1^{\AWB}({\qu})$. The following result shows the existence of covers for arbitrary finite-dimensional \textup{AWB}'s.

\begin{proposition} \label{coverexistence}
Any finite-dimensional \textup{AWB} ${\qu}$ has at least one cover.
\end{proposition}
\begin{proof}
Let $\fe/\re \cong\qu$ be a free presentation of $\qu$ and $\se/[\re, \fe]$ be a complement  of ${\sf H}_1^{\AWB}({\qu})$ in $\re/[\re, \fe]$, for a suitable ideal $\se$ in $\fe$. If we put $\ga = \fe/\se$ and $\ene = \re/\se$, then $\ga/\ene\cong \qu$ and $\ene\cong {\sf H}_1^{\AWB}({\qu})$. From the assumption we have $\re = (\re \cap [[\fe,\fe]]) + \se \subseteq [[\fe,\fe]] + \se$, then $\ene\subseteq{\ze}(\ga)\cap [[\ga,\ga]]$. Hence $\ga$ is a cover of $\qu$.
\end{proof}

\begin{proposition} \label{allstemcovers}
All stem covers of a given \textup{AWB} ${\qu}$ are mutually isoclinic.
\end{proposition}
\begin{proof}
Assume $(G_1)$ and $(G_2)$ are two stem covers of ${\qu}$. It follows from Proposition \ref{stemcover} that ${\Ker}(\theta(G_i))=0$, $i=1,2$. Hence Theorem \ref{main1} implies that $\id_{\qu} : \qu \to \qu$ induces an isoclinism from $(G_1)$ to $(G_2)$, as desired.
\end{proof}

}

\section*{\bf Acknowledgments} First author was supported by Agencia Estatal de Investigaci\'on (Spain), grant PID2020-115155GB-I00 (European FEDER
	support included, UE).

\begin{center}

\end{center}

\end{document}